\input ./style/arxiv-vmsta.cfg
\documentclass[numbers,compress]{vmsta}
\usepackage{dsfont,enumitem,multirow}

\volume{2}
\pubyear{2015}
\firstpage{17}
\lastpage{28}
\doi{10.15559/15-VMSTA21}

\startlocaldefs

\newcommand{\rrvert}{\vert}
\newcommand{\llvert}{\vert}
\allowdisplaybreaks

\newtheorem{theorem}{Theorem}[section]
\newtheorem{lemma}[theorem]{Lemma}
\newtheorem{corollary}[theorem]{Corollary}
\theoremstyle{definition}
\newtheorem{remark}[theorem]{Remark}

\newcommand{\R}{\mathbb R}
\newcommand{\N}{\mathbb N}
\newcommand{\E}{\mathbf E}
\newcommand{\Et}{\E_\theta}
\newcommand{\F}{\mathfrak F}
\newcommand{\pr}{\mathbf P}
\newcommand{\prt}{\pr_\theta}
\newcommand{\ind}{\mathds1}
\newcommand*{\abs}[1]{\big\llvert #1\big\rrvert }

\newcommand{\toL}{\xrightarrow{L_1}}
\DeclareMathOperator{\sgn}{sgn}
\endlocaldefs

\begin{document}
\begin{frontmatter}

\title{Asymptotic normality of discretized maximum likelihood estimator
for drift parameter in~homogeneous diffusion model}
\author[a]{\inits{K.}\fnm{Kostiantyn}\snm{Ralchenko}}\email
{k.ralchenko@gmail.com}
\address[a]{Taras Shevchenko National University of Kyiv, \\
Department of Probability Theory, Statistics and Actuarial Mathematics,
\\
Volodymyrska 64/13, 01601 Kyiv, Ukraine}

%



\markboth{K. Ralchenko}{Asymptotic normality of discretized maximum
likelihood estimator for drift parameter}

\begin{abstract}
We prove the asymptotic normality of the discretized maximum likelihood
estimator for the drift parameter in the homogeneous ergodic diffusion model.
\end{abstract}


\received{17 March 2015}
\revised{4 April 2015}
\accepted{5 April 2015}
\publishedonline{13 April 2015}
\end{frontmatter}

\section{Introduction}
The statistical inference for diffusion models has been thoroughly
studied by now; see the books \cite
{Heyde:1997,Kessler_et_al:2012,Kutoyants:2004,Liptser:Shiryaev:1978:v2,PrakasaRao:1987}
and references therein.

In this paper, we consider the homogeneous diffusion process given by
the stochastic differential equation
\begin{equation*}
dX_t=\theta a(X_t)\,dt+b(X_t)
\,dW_t, 
\end{equation*}
where $W_t$ is a standard Wiener process, and $\theta$ is an unknown
parameter.

The standard maximum likelihood estimator for the parameter $\theta$
constructed by the observations of $X$ on the interval $[0,T]$ has the
form\vadjust{\eject}
\[
\hat\theta_T=\frac{\int_0^T\frac{a(X_t)}{b(X_t)^2}\,dX_t}{\int_0^T\frac
{a(X_t)^2}{b(X_t)^2}\,dt};
\]
see, for instance, \cite[Example 1.37]{Kutoyants:2004} and \cite{Mishura:2014}.
If the equation has a weak solution, the coefficient $a$ is not
identically zero, and the functions
$\frac{1}{b^2}$, $\frac{a^2}{b^2}$, $\frac{a^2}{b^4}$
are locally integrable, then this estimator is strongly consistent
\xch{}{(}\cite[Thm. 3.3]{Mishura:2014}\xch{}{)}.
Moreover, if the model is ergodic, then this estimator is
asymptotically normal \xch{}{(}\cite[Ex.~1.37]{Kutoyants:2004}\xch{}{)}. Note that in
the nonergodic case the maximum likelihood estimator $\hat\theta_T$ may
have different limit distributions; some examples can be found in \cite
[Sect. 3.5]{Kutoyants:2004}.

If the data are the observations of the trajectory $\{X_t,t\ge0\}$ at
discrete time moments $t_1,t_2,\dots$, we obtain the discrete-time
version of the model. Parameter estimation in such models has been
studied since the mid-1980s; see \cite{DC_FZ86,FZ89,PrakasaRao88}.
A review of this problem and many references can be found in \cite
{Iacus:2008} and \cite{HSorensen:2002}. For recent results, see \cite
{Kessler_et_al:2012,Mishura:2014,Sorensen:2009}.

In this paper, we are interested in the scheme of observations that is
called ``rapidly increasing experimental design.''
The process $X$ is observed at time moments\break \mbox{$t_i = i\Delta_n$},
$i=0,\dots,n$, such that
$\Delta_n\to0$ and $n\Delta_n\to\infty$ as $n\to\infty$.
One of possible approaches to parameter estimation is to consider a
discretized version of the continuous-time MLE $\hat\theta_T$.
The most general results in this direction were obtained
by Yoshida \cite{Yoshida92}.
He proved the consistency and asymptotic normality of the discretized
MLE in the model, where the process was multidimensional, the drift
coefficient depended on $\theta$ nonlinearly, and the diffusion
coefficient also contained an unknown parameter.

Assume that we observe the process $X$ at discrete time moments
$t_k^n=k/n$, \mbox{$0\le k\le n^{1+\alpha}$},
where $0<\alpha<\frac{1}2$.
In this scheme, Mishura \cite{Mishura:2014} proposed the following
discretized version of the maximum likelihood estimator:
\[
\hat\theta_n=\frac{\sum_{k=0}^{n^{1+\alpha}}a (X_{\frac{k}{n}} )
 (X_{\frac{k+1}{n}}-X_{\frac{k}{n}} )/b (X_{\frac
{k}{n}} )^2}{n^{-1}\sum_{k=0}^{n^{1+\alpha}}a (X_{\frac
{k}{n}} )^2/b (X_{\frac{k}{n}} )^2}.
\]
She proved its strong consistency in the case where the coefficients
$a$ and $b$ are bounded.
The aim of this paper is to establish the asymptotic normality of this
estimator. Additionally, we assume the ergodicity of the model, but the
boundedness of the coefficients is not required.
In comparison with general results of Yoshida \cite{Yoshida92}, our
assumptions are less restrictive.
We assume the polynomial growth of the function $1/b$ instead of the condition
$\inf_xb(x)^2>0$.
Also, we do not assume the smoothness of the coefficients and the
polynomial growth of their derivatives; any Lipschitz continuous $a(x)$
and $b(x)$ are possible.

The paper is organized as follows.
In Section \ref{sec:result}, we describe the model and formulate the results.
In Section \ref{sec:simul}, some simulation experiments are considered.
The proof of the main theorem is given in
Section~\ref{sec:proofs}.

\section{Model description and main result}\label{sec:result}
Let $(\varOmega,\F)$ be a measurable space.
Assume that $\theta\in\R$ is fixed but unknown.
Consider a probability measure
$\prt$
such that $\F$ is $\prt$-complete.

Let $X$ solve the equation
\begin{equation}
X_t=x_0+\theta\int_0^ta(X_s)\,ds+
\int_0^tb(X_s)\,dW_s,
\label{eq:SDE}
\end{equation}
where $x_0\in\R$, $a,b\colon\R\to\R$ are measurable functions, and
$\{W_t,t\geq0\}$ is a~standard Wiener
process on $(\varOmega,\F,\prt)$.

Denote
$c(x)=\frac{a(x)}{b(x)^2}$,
$d(x)=\frac{a(x)^2}{b(x)^2}$,
$\varphi_\theta(x)=\exp \{-2\theta\int_0^xc(y)\,dy \}$,
and
$\varPhi_\theta(x)=\int_0^x\varphi_\theta(y)\,dy$.

Assume that the following conditions hold.

\begin{enumerate}[label=(A\arabic*)]
\item\label{as:lipschitz}
For some $L>0$ and for any $x,y\in\R$,
\[
\abs{a(x)-a(y)}+\abs{b(x)-b(y)}\le L|x-y|.
\]
\item\label{as:reccurence}
$\varPhi_\theta(+\infty)=-\varPhi_\theta(-\infty)=+\infty$.
\item\label{as:positive_rec} $G_\theta:=\int_{-\infty}^{+\infty}\frac
{dx}{b(x)^2\varphi_\theta(x)}<\infty$.
\end{enumerate}

It is well known that under assumption \ref{as:lipschitz} the
stochastic differential equation \eqref{eq:SDE} has a unique strong solution.
This assumption also yields that the functions $a(x)$ and $b(x)$
satisfy the linear growth condition, that is,
\begin{equation}
\abs{a(x)}+\abs{b(x)}\le M_1\bigl(1+|x|\bigr) \label{eq:lin_growth}
\end{equation}
for some $M>0$ and for all $x\in\R$.

Assume additionally that
\begin{enumerate}[resume*]
\item\label{as:b>0}
\xch{There}{there} exist $K>0$ and $p\ge0$ such that
\[
\abs{b(x)}^{-1}\le K\bigl(1+|x|^p\bigr).
\]
\end{enumerate}
Then, for some $M_2>0$ and for any $x\in\R$,
\begin{equation}
\abs{c(x)}\le M_2 \bigl(1+|x|^{2p+1} \bigr), \qquad
\abs{d(x)}\le M_2 \bigl(1+|x|^{2p+2} \bigr). \label{eq:growth}
\end{equation}

Under assumptions \ref{as:reccurence}--\ref{as:positive_rec}, the diffusion
process $X$ is positive recurrent; see, for example,\ \cite[Prop.
1.15]{Kutoyants:2004}.
In this case, it has ergodic properties with the invariant density
given by
\begin{equation*}
\mu_\theta(x)=\frac{1}{G_\theta b(x)^2\varphi_\theta(x)}, \quad x\in\R. 
\end{equation*}
Let $\xi_\theta$ denote a random variable with density $\mu_\theta(x)$.
Then, for any measurable function $h$ such that
$\Et|h(\xi_\theta)|<\infty$,
\begin{equation}
\frac{1}T\int_0^Th(X_t)
\,dt\to \int_{-\infty}^{+\infty}h(x)\mu_\theta(x)\,dx
\equiv\Et h(\xi_\theta) \quad\text{a.s.\ as }T\to\infty, \label{eq:cont-LLN}
\end{equation}
see \cite[Thm. 1.16]{Kutoyants:2004}.
Moreover, according to \cite[Sect. II.37]{BorodinSalminen}, the
convergence \eqref{eq:cont-LLN} holds also in $L_1$, that is,
\begin{equation}
\frac{1}T\Et\int_0^Th(X_t)
\,dt\to \Et h(\xi_\theta). \label{eq:cont-LLN-L1}
\end{equation}
Assume that the invariant distribution satisfies the condition
\begin{enumerate}[resume*]
\item\label{as:4moment}
$\Et|\xi_\theta|^{r}
\equiv\int_{-\infty}^{+\infty}|x|^r\mu_\theta(x)\,dx
<\infty$ for all $r\ge0$.\vadjust{\eject}
\end{enumerate}

Let $0<\alpha<1$.
Suppose that we observe the process $X$ at discrete time moments
$t_k^n=k/n$, $0\le k\le n^{1+\alpha}$.
Consider the estimator
\[
\hat\theta_n=\frac{\sum_{k=1}^{n^{1+\alpha}}c (X_{\frac
{k-1}{n}} )\Delta X_k^n}%
{n^{-1}\sum_{k=1}^{n^{1+\alpha}}d (X_{\frac{k-1}{n}} )},
\]
where
$\Delta X_k^n=X_{\frac{k}{n}}-X_{\frac{k-1}{n}}$.

Assume also that
\begin{enumerate}[resume*]
\item\label{as:FI>0}
$a$ is not identically zero.
\end{enumerate}

Then $\Et d(\xi_\theta)>0$.
Note also that by \eqref{eq:growth} and \ref{as:4moment},
$\Et d(\xi_\theta)<\infty$.
Now we are ready to formulate the main result.
\renewcommand{\theenumi}{\roman{enumi}}%
\begin{theorem}\label{thm}
Assume that conditions \emph{\ref{as:lipschitz}--\ref{as:FI>0}} hold.
Then
\begin{enumerate}
\item[\rm(i)]
$\hat\theta_n\xrightarrow\prt\theta$
as $n\to\infty$,
\item[\rm(ii)]
$n^{\alpha/2} (\hat\theta_n-\theta )
\Rightarrow N (0,1/\Et d(\xi_\theta) )$
as $n\to\infty$.
\end{enumerate}
\end{theorem}
The proof is given in Section \ref{sec:proofs}.

The following result gives sufficient conditions for consistency and asymptotic
normality in the case where the parameter $\theta$ is positive.

\begin{corollary}
Let $\theta>0$.
Assume that conditions \emph{\ref{as:lipschitz}}, \emph{\ref{as:b>0}}
and \emph{\ref{as:FI>0}}
are fulfilled and, additionally,
\begin{equation}
\limsup_{|x|\to\infty}c(x)\sgn(x)<0. \label{eq:A0}
\end{equation}
Then
\begin{enumerate}
\item[\rm(i)]
$\hat\theta_n\xrightarrow\prt\theta$
as $n\to\infty$,
\item[\rm(ii)]
$n^{\alpha/2} (\hat\theta_n-\theta )
\Rightarrow N (0,1/\Et d(\xi_\theta) )$
as $n\to\infty$.
\end{enumerate}
\end{corollary}
\begin{proof}
Note that condition~\eqref{eq:A0}, together with \ref{as:b>0}, implies that
assumptions \ref{as:reccurence}--\ref{as:positive_rec} are satisfied
and, moreover,
all polynomial moments of the invariant density are finite; see
\cite[p.~3]{Kutoyants:2004}.
Hence, the result follows directly from Theorem~\ref{thm}.
\end{proof}

If the coefficients are bounded, then the consistency and asymptotic
normality of $\hat\theta_n$ can be obtained without assumption \ref{as:4moment}.

\begin{corollary}\label{cor:2}
Assume that conditions \emph{\ref{as:lipschitz}--\ref{as:positive_rec}}
are satisfied, the coefficients $a(x)$ and $b(x)$ are bounded, and
$\inf_{x\in\R}|b(x)|>0$.
Then
\begin{enumerate}
\item[\rm(i)]\label{cor:2:i}
$\hat\theta_n\xrightarrow\prt\theta$
as $n\to\infty$,
\item[\rm(ii)]
$n^{\alpha/2} (\hat\theta_n-\theta )
\Rightarrow N (0,1/\Et d(\xi_\theta) )$
as $n\to\infty$.
\end{enumerate}
\end{corollary}
\begin{proof}[Sketch of proof]
This result can be proved similarly to Theorem~\ref{thm} using the
boundedness of $a(x)$, $b(x)$, $c(x)$, $d(x)$ instead of the growth
conditions \eqref{eq:lin_growth}, \eqref{eq:growth}, and
\ref{as:b>0}\vadjust{\eject}
together with the boundedness of moments of the invariant density. In
this case, \eqref{eq:upbound1}~implies the inequality
\[
\Et (X_t-X_{\frac{k-1}{n}} )^{2m}\le C(m,
\theta)n^{-m}
\]
for all $m\in\N$ and $t\in[\frac{k-1}n,\frac{k}n]$, $k=1,2,\dots n^\alpha$.
This estimate is used in the proof instead of Lemmas~\ref
{l:bound1}--\ref{l:bound2}.
\end{proof}

\begin{remark}
For $\alpha\in(0,\frac{1}2)$, Mishura \cite{Mishura:2014} obtained the
a.s.\ convergence in Corollary~\ref{cor:2}(i) \xch{}{\ref{cor:2:i}} without
assumptions \ref{as:reccurence}--\ref{as:positive_rec}.
\end{remark}

\section{Some simulation results}\label{sec:simul}
In this section, we illustrate quality of the estimator by simulation
experiments.
We consider the diffusion process \eqref{eq:SDE} with drift parameter
$\theta=2$ and initial value $x_0=1$ in three following cases:
\begin{enumerate}[label=(\arabic*)]
\item$a(x) =1-x$, $b(x)=2+\sin x$,
\item$a(x) =-\arctan x$, $b(x)=1$,
\item$a(x) =-\frac{x}{1+x^2}$, $b(x)=1$.
\end{enumerate}
Using the Milstein method, we simulate 100 sample paths of each process
and find the estimate $\hat\theta_n$
for different values of $n$ and $\alpha$.
The average values of $\hat\theta_n$ and the corresponding standard
deviations are presented in Tables \ref{tab1}--\ref{tab3}.

\begin{table}[b]
\vspace*{6pt}
\caption{\label{tab1}$a(x) =1-x$, $b(x)=2+\sin x$}
\begin{tabular}{ccllllll}\hline
\multicolumn{2}{c}{} & \multicolumn{6}{l}{$n$}\\\cline{3-8}
\multicolumn{2}{c}{} & 50 & 100 & 500 & 1000 & 2000 & 5000 \\\hline
$\alpha=0.1$ & Mean & 3.05812 & 2.97626 & 2.73973 &
2.58453 & 2.55888 & 2.53879 \\\smallskip
& Std.\,dev. & 2.06388 & 2.00007 & 1.43273 & 1.34689 & 1.26920 &
1.22077 \\
$\alpha=0.5$ & Mean & 2.11065 & 2.15066 & 2.08157 &
2.05626 & 2.03686 & 2.03479 \\\smallskip
& Std.\,dev. & 0.62613 & 0.56038 & 0.31621 & 0.28909 & 0.22875 &
0.18187 \\
$\alpha=0.9$ & Mean & 2.02509 & 2.01702 & 2.02024 &
2.01308 & 2.00626 & 2.00289 \\
& Std.\,dev. & 0.27874 & 0.19589 & 0.09995 & 0.06918 & 0.04850 &
0.03028\\\hline
\end{tabular}\vspace*{10pt}
\end{table}

\begin{table}[b!]
\caption{\label{tab2}$a(x) =-\arctan x$, $b(x)=1$}
\begin{tabular}{ccllllll}\hline
\multicolumn{2}{c}{} & \multicolumn{6}{l}{$n$}\\\cline{3-8}
\multicolumn{2}{c}{} & 50 & 100 & 500 & 1000 & 2000 & 5000 \\\hline
$\alpha=0.1$ & Mean & 2.69321 & 2.66637 & 2.65053 &
2.66356 & 2.59903 & 2.46685 \\\smallskip
& Std.\,dev. & 2.03142 & 2.06075 & 1.82903 & 1.73034 & 1.68212 &
1.50186 \\
$\alpha=0.5$ & Mean & 2.12190 & 2.10459 & 2.01048 &
1.99535 & 2.01712 & 1.99517 \\\smallskip
& Std.\,dev. & 0.85304 & 0.69484 & 0.48803 & 0.37807 & 0.31746 &
0.25846 \\
$\alpha=0.9$ & Mean & 1.95538 & 1.97446 & 1.98035 &
1.99565 & 2.00266 & 2.00290 \\
& Std.\,dev. & 0.35057 & 0.26796 & 0.12235 & 0.09050 & 0.06496 &
0.04533 \\\hline
\end{tabular}
\end{table}

\begin{table}
\caption{\label{tab3}$a(x) =-\frac{x}{1+x^2}$, $b(x)=1$}
\begin{tabular}{ccllllll}\hline
\multicolumn{2}{c}{} & \multicolumn{6}{l}{$n$}\\\cline{3-8}
\multicolumn{2}{c}{} & 50 & 100 & 500 & 1000 & 2000 & 5000 \\\hline
$\alpha=0.1$ & Mean & 1.99507 & 1.99813 & 1.97122 &
1.99255 & 1.98366 & 1.94811 \\\smallskip
& Std.\,dev. & 2.44248 & 2.53060 & 2.17322 & 2.13403 & 2.05527 &
1.80128\\
$\alpha=0.5$ & Mean & 1.87038 & 1.87897 & 1.89022 &
1.92593 & 1.94964 & 1.96624 \\\smallskip
& Std.\,dev. & 1.01932 & 0.89315 & 0.54811 & 0.49005 & 0.41787 &
0.33855\\
$\alpha=0.9$ & Mean & 1.90341 & 1.92162 & 2.00240 &
2.00068 & 2.00491 & 1.99347 \\
& Std.\,dev. & 0.47656 & 0.33693 & 0.18136 & 0.13173 & 0.09595 &
0.07033 \\\hline
\end{tabular}
\end{table}

\section{Proof of Theorem~\ref{thm}}\label{sec:proofs}
In this section, we prove the main theorem and some auxiliary lemmas.
In what follows, $C,C_1,C_2,\dots$ are positive generic constants that
may vary from line to line. If they depend on some arguments, we will
write $C(\theta)$, $C(m,\theta)$, and so on.

By~\eqref{eq:SDE},
\begin{align*}
\Delta X_k^n&=\theta\int_{\frac{k-1}{n}}^{\frac{k}n}a(X_t)
\,dt +\int_{\frac{k-1}{n}}^{\frac{k}n}b(X_t)
\,dW_t
\\
&=\theta a (X_{\frac{k-1}{n}} )\frac{1}n +\theta\int
_{\frac{k-1}{n}}^{\frac{k}n} \bigl(a(X_t)-a
(X_{\frac{k-1}{n}
} ) \bigr)\,dt +b(X_{\frac{k-1}{n}})\Delta W_k^n
\\
&\quad+\int_{\frac{k-1}{n}}^{ \frac{k}n} \bigl(b(X_t)-b
(X_{\frac
{k-1}{n}} ) \bigr)\,dW_t.
\end{align*}
Therefore,
\begin{align*}
\hat\theta_n&=\theta+ \sum_{k=1}^{n^{1+\alpha}}
\Biggl(c (X_{\frac{k-1}{n}} )\theta \int_{\frac{k-1}{n}}^{\frac{k}n}
\bigl(a(X_t)-a (X_{\frac{k-1}{n}} ) \bigr)\,dt +\frac{a (X_{\frac{k-1}{n}} )}{b (X_{\frac{k-1}{n}} )} \Delta
W_k^n
\\
&\quad+ c (X_{\frac{k-1}{n}} )\int_{\frac{k-1}{n}}^{\frac{k}n}
\bigl(b(X_t)-b (X_{\frac{k-1}{n}} ) \bigr)\,dW_t \Biggr)
\Bigm/ \Biggl(\frac{1}{n}\sum_{k=1}^{n^{1+\alpha}}d
(X_{\frac
{k-1}{n}} ) \Biggr).
\end{align*}
Then
\begin{equation}
\hat\theta_n-\theta =\frac{n^{-\alpha}(A_n+B_n+E_n)}{D_n}, \label{eq:differ}
\end{equation}
where
\begin{align*}
D_n&=n^{-1-\alpha}\sum_{k=1}^{n^{1+\alpha}}d
(X_{\frac
{k-1}{n}} ),
\\
A_n&=\sum_{k=1}^{n^{1+\alpha}}c
(X_{\frac{k-1}{n}} )\theta \int_{\frac{k-1}{n}}^{\frac{k}n}
\bigl(a(X_t)-a (X_{\frac{k-1}{n}} ) \bigr)\,dt,
\\
B_n&=\sum_{k=1}^{n^{1+\alpha}}
\frac{a (X_{\frac{k-1}{n}} )}%
{b (X_{\frac{k-1}{n}} )}\Delta W_k^n,
\\
E_n&=\sum_{k=1}^{n^{1+\alpha}}c
(X_{\frac{k-1}{n}} ) \int_{\frac{k-1}{n}}^{\frac{k}n}
\bigl(b(X_t)-b (X_{\frac{k-1}{n}} ) \bigr)\,dW_t.
\end{align*}

\begin{lemma}\label{l:bound1}
Let assumptions \emph{\ref{as:lipschitz}--\ref{as:positive_rec}} and
\emph{\ref{as:4moment}} be fulfilled.
Then for every $m\in\N$, there exists a constant $C(m,\theta)>0$ such that
\begin{align*}
\Et (X_t-X_{\frac{k-1}{n}} )^{2m}&\le C(m,
\theta)n^{-m+1+\alpha}.
\end{align*}
for all $n\in\N$, $1\le k\le n^{1+\alpha}$, and $t\in [\frac
{k-1}{n},\frac{k}{n} ]$.
\end{lemma}

\begin{proof}
By~\eqref{eq:SDE} and the inequality
$(a+b)^{2m}\le2^{2m-1}(a^{2m}+b^{2m})$,
\begin{align*}
&\Et (X_t-X_{\frac{k-1}{n}} )^{2m}\\
&\quad\le2^{2m-1} \Biggl(\theta^{2m}\Et \Biggl(\int
_{\frac{k-1}n}^ta(X_s)\, ds
\Biggr)^{2m} +\Et \Biggl(\int_{\frac{k-1}n}^tb(X_s)
\,dW_s \Biggr)^{2m} \Biggr).
\end{align*}
Using the Burkholder--Davis--Gundy inequality, we obtain
\begin{align*}
&\Et (X_t-X_{\frac{k-1}{n}} )^{2m}\\
&\quad\le2^{2m-1} \Biggl(\theta^{2m}\Et \Biggl(\int
_{\frac{k-1}n}^ta(X_s)\, ds
\Biggr)^{2m} +C(m)\Et \Biggl(\int_{\frac{k-1}n}^tb(X_s)^2
\,ds \Biggr)^m \Biggr).
\end{align*}
By Jensen's inequality,
\begin{align}
 \Et (X_t-X_{\frac{k-1}{n}}
)^{2m}&\le2^{2m-1} \Biggl(\theta^{2m} \bigl(t-\tfrac{k-1}n
\bigr)^{2m-1} \Et\int_{\frac{k-1}n}^ta(X_s)^{2m}
\,ds\notag\\
&\quad+ C(m) \bigl(t-\tfrac{k-1}n \bigr)^{m-1}\Et\int
_{\frac
{k-1}n}^tb(X_s)^{2m}\,ds
\Biggr).\label{eq:upbound1}
\end{align}
Further,
we have
\begin{align*}
\Et (X_t-X_{\frac{k-1}{n}} )^{2m} &\le2^{2m-1}
\Biggl(\theta^{2m}n^{1-2m} \Et\int_0^{n^\alpha}a(X_s)^{2m}
\,ds
\\
&\quad+ C(m)n^{1-m}\Et\int_0^{n^\alpha}b(X_s)^{2m}
\,ds \Biggr).
\end{align*}
Now it remains to note that by~\eqref{eq:lin_growth} and \eqref
{eq:cont-LLN-L1} the integrals
$n^{-\alpha}\int_0^{n^\alpha}a(X_s)^{2m}\,ds$
and
$n^{-\alpha}\int_0^{n^\alpha}b(X_s)^{2m}\,ds$
have bounded expectations.
\end{proof}

\begin{lemma}\label{l:bound2}
Under assumption \emph{\ref{as:lipschitz}},
for every $m\in\N$, there exists a constant $C(m,\theta)>0$ such that
\[
\Et (X_t-X_{\frac{k-1}{n}} )^{2m}\le C(m,
\theta)n^{-m}\Et \bigl(1+|X_{\frac{k-1}{n}}|^{2m} \bigr)
\]
for all $n\in\N$, $1\le k\le n^{1+\alpha}$, and $t\in [\frac
{k-1}{n},\frac{k}{n} ]$.
\end{lemma}

\begin{proof}
By \eqref{eq:upbound1},
\begin{align*}
&\Et (X_t-X_{\frac{k-1}{n}} )^{2m}
\\
&\quad\le C_1(m,\theta)n^{1-m} \Biggl(\Et\int
_{\frac{k-1}n}^ta(X_s)^{2m}\,ds +\Et
\int_{\frac{k-1}n}^tb(X_s)^{2m}\,ds
\Biggr).
\end{align*}
Using assumption \ref{as:lipschitz} and \eqref{eq:lin_growth}, we get
\begin{align*}
\Et\int_{\frac{k-1}n}^ta(X_s)^{2m}
\,ds&\le2^{2m-1}\Et\int_{\frac{k-1}n}^t \bigl(
\bigl(a(X_s)-a (X_{\frac
{k-1}n} ) \bigr)^{2m} +a
(X_{\frac{k-1}n} )^{2m} \bigr)\,ds\\
&\le2^{2m-1}L\Et\int_{\frac{k-1}n}^t
(X_s-X_{\frac{k-1}n} )^{2m}\,ds\\
&\quad+2^{2m-1}M \bigl(t-\tfrac{k-1}{n} \bigr)\Et \bigl(1+
|X_{\frac{k-1}n}|^{2m} \bigr).
\end{align*}
The same estimate holds for
$\Et\int_{\frac{k-1}n}^tb(X_s)^{2m}\,ds$.
Therefore,
\begin{align*}
\Et (X_t-X_{\frac{k-1}{n}} )^{2m} &\le C_2(m,
\theta)n^{1-m}\int_{\frac{k-1}n}^t\Et
(X_s-X_{\frac
{k-1}n} )^{2m}\,ds
\\
&\quad+C_2(m,\theta)n^{-m}\Et \bigl(1+|X_{\frac{k-1}n}|^{2m}
\bigr),
\end{align*}
and the result follows from the Gronwall lemma.
\end{proof}

\begin{lemma}\label{l:bound_prod}
Assume that conditions ~\emph{\ref{as:lipschitz}--\ref
{as:positive_rec}} and \emph{\ref{as:4moment}}
are fulfilled.
Then for any $m\ge1$, $1\le i\le2m$, and $0\le j\le2m$, there exists
$C(m,\theta)>0$
such that
\[
\sum_{k=1}^{n^{1+\alpha}}\int_{\frac{k-1}{n}}^{\frac{k} n}
\Et \bigl(|X_{\frac{k-1}{n}}-X_t|^i
|X_t|^j \bigr) dt \le C(m,\theta)n^{\alpha-\frac{i(m-\alpha)}{2m+2}}.
\]
\end{lemma}
\begin{proof}
Applying the H\"older inequality and Lemma~\ref{l:bound1}, we get
\begin{align*}
\Et \bigl(|X_{\frac{k-1}{n}}-X_t|^i
|X_t|^j \bigr) &\le \bigl(\Et|X_{\frac{k-1}{n}}-X_t|^{2m+2}
\bigr)^{\frac{i}{2m+2}} \bigl(\Et|X_t|^{\frac{j(2m+2)}{2m+2-i}}
\bigr)^{\frac
{2m+2-i}{2m+2}}
\\
&\le C_1(m,\theta) n^{-\frac{(m-\alpha)i}{2m+2}} \bigl(\Et|X_t|^{\frac{j(2m+2)}{2m+2-i}}
\bigr)^{\frac{2m+2-i}{2m+2}}.
\end{align*}
Then
\begin{align*}
&\sum_{k=1}^{n^{1+\alpha}}\int_{\frac{k-1}{n}}^{\frac{k} n}
\Et \bigl(|X_{\frac{k-1}{n}}-X_t|^i
|X_t|^j \bigr) \,dt
\\
&\quad\le C_1(m,\theta)n^{-\frac{(m-\alpha)i}{2m+2}} \int_0^{n^\alpha}
\bigl(\Et|X_t|^{\frac{j(2m+2)}{2m+2-i}} \bigr)^{\frac{2m+2-i}{2m+2}}\,dt.
\end{align*}
By Jensen's inequality we have
\begin{equation*}
\int_0^{n^\alpha} \bigl(\Et|X_t|^{\frac{j(2m+2)}{2m+2-i}}
\bigr)^{\frac{2m+2-i}{2m+2}}\,dt \le n^{\alpha} \Biggl(n^{-\alpha}\int
_0^{n^\alpha} \Et|X_t|^{\frac{j(2m+2)}{2m+2-i}}\,dt
\Biggr)^{\frac{2m+2}{2m+2-i}}.
\end{equation*}
By \eqref{eq:cont-LLN-L1} the expression in brackets is bounded.
This completes the proof.
\end{proof}

\begin{lemma}\label{discr-LLN}
Assume that conditions~\emph{\ref{as:lipschitz}--\ref{as:positive_rec}}
and \emph{\ref{as:4moment}}
are fulfilled.
Then
\begin{enumerate}
\item[\rm(i)]\label{discr-LLN-i}
for any $m=0,1,2,\ldots,$
\[
n^{-1-\alpha}\sum_{k=1}^{n^{1+\alpha}}X_{\frac{k-1}{n}}^{2m}
\toL\Et\xi_\theta^{2m} \quad\text{as }n\to\infty;
\]
\item[\rm(ii)]\label{discr-LLN-ii}
if, additionally, \emph{\ref{as:b>0}} holds, then
\[
D_n\toL\Et d(\xi_\theta)\quad\text{as }n\to\infty.
\]
\end{enumerate}
\end{lemma}

\begin{proof}
(i)
In the case $m=0$, the result is trivial.
Let $m\ge1$.
By~\eqref{eq:cont-LLN-L1} we have
\[
n^{-\alpha}\int_0^{n^\alpha}X_t^{2m}
\,dt\toL\Et\xi_\theta^{2m} \quad\text{as }n\to\infty.
\]
Hence, it suffices to prove that
\begin{align*}
F_n&:=\Et\bigg|n^{-\alpha}\int_0^{n^\alpha}X_t^{2m}
\,dt -n^{-1-\alpha}\sum_{k=1}^{n^{1+\alpha}}X_{\frac{k-1}{n}}^{2m}\bigg|
\\
&=n^{-\alpha}\Et\bigg|\sum_{k=1}^{n^{1+\alpha}}
\int_{\frac
{k-1}{n}}^{\frac{k}{n}} \bigl(X_t^{2m}-X_{\frac{k-1}{n}}^{2m}
\bigr)\,dt\bigg|
\end{align*}
converges to zero as $n\to\infty$.
By the inequality
$|x|\le|x-y|+|y|$,
\begin{equation}
\abs{x^{2m}-y^{2m}} \le|x-y|\sum
_{i=0}^{2m-1}|x|^i|y|^{2m-1-i}
\le\sum_{i=1}^{2m}C_i
|x-y|^i|y|^{2m-i}. \label{eq:bound_pow}
\end{equation}
Therefore,
\[
F_n\le\sum_{i=1}^{2m}C_i
n^{-\alpha}\sum_{k=1}^{n^{1+\alpha}} \int
_{\frac{k-1}{n}}^{\frac{k}{n}}\Et \bigl(|X_{\frac{k-1}{n}}-X_t|^i
|X_t|^{2m-i} \bigr)\,dt,
\]
and, by Lemma~\ref{l:bound_prod},
\[
F_n\le C(m,\theta)\sum_{i=1}^{2m}C_i
n^{-\frac{i(m-\alpha)}{2m+2}} \to0\quad\text{as }n\to\infty.
\]

 (ii)
For arbitrary $x$ and $y$,
\begin{align*}
d(x)-d(y) &=\frac{a(x)^2}{b(x)^2}-\frac{a(y)^2}{b(y)^2}
\\
&=\bigl(a(x)-a(y)\bigr) \biggl(\frac{a(x)}{b(x)^2}+\frac{a(y)}{b(x)b(y)} \biggr)
\\
&\quad -\bigl(b(x)-b(y)\bigr) \biggl(\frac{a(x)a(y)}{b(x)^2b(y)}+\frac
{a(x)a(y)}{b(x)b(y)^2}
\biggr).
\end{align*}
By \ref{as:lipschitz}, \ref{as:b>0}, and
\eqref{eq:lin_growth},\vadjust{\eject}
\begin{align}
\abs{d(x)-d(y)}&\le C|x-y| \bigl(1+|x|^{2p+1} + \bigl(1+
|x|^{p} \bigr) \bigl(1+|y|^{p+1} \bigr)
\notag\\
&\quad+ \bigl(1+|x|^{2p+1} \bigr) \bigl(1+|y|^{p+1} \bigr) +
\bigl(1+|x|^{p+1} \bigr) \bigl(1+|y|^{2p+1} \bigr)
\bigr). \label{eq:bound_d}
\end{align}
The rest of the proof can be done similarly to part (i) using
estimate \eqref{eq:bound_d} instead of~\eqref{eq:bound_pow}.
\end{proof}

\begin{lemma}\label{l:conv_nom1&3}
Under the assumptions of Theorem~\ref{thm},
\begin{enumerate}
\item[\rm(i)]\label{conv-1}
$n^{-\alpha/2}|A_n|\xrightarrow{\prt}0$ as $n\to\infty$,
\item[\rm(ii)]\label{conv-2}
$n^{-\alpha/2}|E_n|\xrightarrow{\prt}0$ as $n\to\infty$.
\end{enumerate}
\end{lemma}

\begin{proof}
 (i)
By the Cauchy--Schwarz inequality we have
\begin{align*}
\Et|A_n|&\le|\theta|\sum_{k=1}^{n^{1+\alpha}}
\int_{\frac{k-1}{n}}^{\frac{k}n}\Et\abs{c (X_{\frac{k-1}{n}} )
\bigl(a(X_t)-a (X_{\frac{k-1}{n}} ) \bigr)}\,dt
\\
&\le|\theta|\sum_{k=1}^{n^{1+\alpha}} \int
_{\frac{k-1}{n}}^{\frac{k}n} \bigl(\Et c (X_{\frac{k-1}{n}}
)^2 \bigr)^{\frac{1}2} \bigl(\Et \bigl(a(X_t)-a
(X_{\frac{k-1}{n}} ) \bigr)^2 \bigr)^{\frac{1}2}\,dt.
\end{align*}
Using~\ref{as:lipschitz}, \eqref{eq:growth}, and Lemma~\ref{l:bound2},
we get
\begin{align*}
\Et|A_n|&\le C_1 |\theta|\sum
_{k=1}^{n^{1+\alpha}} \int_{\frac{k-1}{n}}^{\frac{k}n}
\bigl(\Et \bigl(1+|X_{\frac{k-1}{n}}|^{4p+2} \bigr)
\bigr)^{\frac{1}2} \bigl(\Et (X_t-X_{\frac{k-1}{n}} )^2
\bigr)^{\frac{1}2}\,dt
\\
&\le C_2(\theta)n^{-1/2} \sum_{k=1}^{n^{1+\alpha}}
\int_{\frac{k-1}{n}}^{\frac{k}n} \bigl(\Et \bigl(1+
|X_{\frac{k-1}{n}}|^{4p+4} \bigr) \bigr)^{1/2}\, dt
\\
&=C_2(\theta)n^{-3/2} \sum_{k=1}^{n^{1+\alpha}}
\bigl(\Et \bigl(1+|X_{\frac{k-1}{n}}|^{4p+4} \bigr)
\bigr)^{1/2}
\\
&\le C_2(\theta)n^{\alpha-1/2} \Biggl(n^{-1-\alpha}\sum
_{k=1}^{n^{1+\alpha}} \Et \bigl(1+|X_{\frac{k-1}{n}}|^{4p+4}
\bigr) \Biggr)^{1/2}.
\end{align*}
By Lemma~\ref{discr-LLN} the expression
$n^{-1-\alpha}\sum_{k=1}^{n^{1+\alpha}}
\Et (1+|X_{\frac{k-1}{n}}|^{4p+4} )$
is bounded.
Therefore,
\[
n^{-\alpha/2}\Et|A_n|\le C_3(\theta)
n^{\frac{1}2(\alpha-1)} \to0\quad\text{as }n\to\infty.
\]

 (ii)
We have
\[
E_n=\int_0^{n^\alpha}h(t,\omega)
\,dW_t,
\]
where
\[
h(t,\omega)=\sum_{k=1}^{n^{1+\alpha}} c
(X_{\frac{k-1}{n}} ) \bigl(b(X_t)-b (X_{\frac{k-1}{n}} ) \bigr)
\ind_{ [\frac{k-1}{n},\frac{k}{n} )}(t).
\]
Then
\begin{align*}
\Et E_n^2=\Et\int_0^{n^\alpha}h(t,
\omega)^2\,dt =\sum_{k=1}^{n^{1+\alpha}}
\int_{\frac{k-1}{n}}^{\frac{k}n} \Et \bigl(c (X_{\frac{k-1}{n}}
)^2 \bigl(b(X_t)-b (X_{\frac{k-1}{n}} )
\bigr)^2 \bigr)\,dt.
\end{align*}
Similarly to (i), we can estimate
$\Et E_n^2\le C_4(\theta) n^{\alpha-1}$.
Therefore,
$n^{-\alpha/2}|E_n|\xrightarrow{L_2}0$ as $n\to\infty$.
\end{proof}

\begin{lemma}\label{l:conv_denom}
Under the assumptions of Theorem~\ref{thm},
\begin{enumerate}
\item[\rm(i)]
$n^{-\alpha}B_n\xrightarrow{\prt}0$ as $n\to\infty$,
\item[\rm(ii)]
$n^{-\alpha/2}B_n\Rightarrow N (0,\Et d(\xi_\theta) )$
as $n\to\infty$.
\end{enumerate}
\end{lemma}

\begin{proof}
(i)
Let us prove the convergence in $L_2$.
We have
\[
B_n=B_n'+ B_n'',
\]
where
\[
B_n'=\int_0^{n^\alpha}
\frac{a(X_t)}{b(X_t)}\,dW_t, \qquad B_n''=
\sum_{k=1}^{n^{1+\alpha}}\int_{\frac{k-1}{n}}^{\frac{k}{n}}
\biggl(\frac{a (X_{\frac{k-1}{n}} )}{b (X_{\frac
{k-1}{n}} )} -\frac{a(X_t)}{b(X_t)} \biggr)\,dW_t.
\]
Then by \eqref{eq:cont-LLN-L1} we have
\begin{equation}
\Et \bigl(n^{-\alpha/2}B_n' \bigr)^2
=n^{-\alpha}\Et\int_0^{n^\alpha}\frac{a(X_t)^2}{b(X_t)^2}
\,dt =n^{-\alpha}\Et\int_0^{n^\alpha}d(X_t)
\,dt \to\Et d(\xi_\theta) \label{eq:B'}
\end{equation}
as $n\to\infty$.
Hence,
$n^{-\alpha}B_n'\xrightarrow{L_2}0$
as $n\to\infty$.

Arguing as in the proof of Lemma
\ref{l:conv_nom1&3} (ii), we obtain
\[
\Et\bigl(B_n''\bigr)^2=\sum
_{k=1}^{n^{1+\alpha}}\int_{\frac{k-1}{n}}^{\frac{k}{n}}
\biggl(\frac{a (X_{\frac{k-1}{n}} )}{b (X_{\frac
{k-1}{n}} )} -\frac{a(X_t)}{b(X_t)} \biggr)^2\,dt.
\]
Further,
\[
\frac{a(x)}{b(x)}-\frac{a(y)}{b(y)} =\frac{a(x)}{b(x)b(y)}\bigl(b(y)-b(x)\bigr)+
\frac{1}{b(y)}\bigl(a(x)-a(y)\bigr).
\]
Therefore, by~\eqref{eq:lin_growth} and assumption~\ref{as:b>0},
\begin{align*}
\frac{a(x)}{b(x)}-\frac{a(y)}{b(y)} &\le\frac{2a(x)^2}{b(x)^2b(y)^2}\bigl(b(y)-b(x)
\bigr)^2+\frac
{2}{b(y)^2}\bigl(a(x)-a(y)\bigr)^2
\\
&\le C(x-y)^2 \bigl(1+|x|^{2p+2}|y|^{2p}+
|y|^{2p} \bigr).
\end{align*}
Similarly to the proof of Lemma \ref{discr-LLN}, we get the convergence
\begin{equation}
n^{-\alpha}\Et\bigl(B_n''
\bigr)^2\to0 \quad\text{as }n\to\infty. \label{eq:B''}
\end{equation}

(ii) According to \cite[Theorem 1.19]{Kutoyants:2004},
it follows from~\eqref{eq:B'} that
$n^{-\alpha/2}B_n'\Rightarrow N (0,\break\Et d(\xi_\theta) )$
as $n\to\infty$.
Taking into account the convergence~\eqref{eq:B''}, we obtain the result.
\end{proof}

Now the statement of Theorem~\ref{thm} follows from~\eqref{eq:differ}
and Lemmas \ref{discr-LLN}--\ref{l:conv_denom}.

\section*{Acknowledgments}
The author is grateful to the anonymous referee for his careful reading
of this paper and suggesting a number of improvements.

\end{document}